\documentclass[letterpaper, 10 pt, conference]{ieeeconf}  % Comment this line out if you need a4paper

\IEEEoverridecommandlockouts                              
\usepackage{cite}
\usepackage{amsmath,amssymb,amsfonts}
\usepackage{algorithmic}
\usepackage{graphicx}
\usepackage{textcomp}
\usepackage{xcolor}
\usepackage{defs}
\usepackage{caption}
\usepackage{subcaption}
\usepackage{todonotes}
\usepackage{multicol}
\usepackage{tabularray}
\usepackage{mathtools}
\usepackage{todonotes}
\usepackage{float}
\usepackage{booktabs}
\usepackage{array, multirow}
\usepackage[numbered,framed]{matlab-prettifier}

\usepackage{tikz}
\def\BibTeX{{\rm B\kern-.05em{\sc i\kern-.025em b}\kern-.08em
    T\kern-.1667em\lower.7ex\hbox{E}\kern-.125emX}}

\title{\LARGE \bf
A novel trajectory optimization algorithm for continuous-time model predictive control}

\author{Souvik Das$^{1}$, Siddhartha Ganguly$^{1}$, Muthyala Anjali$^{2}$, and Debasish Chatterjee$^{1}$% <-this % stops a space
\thanks{*Souvik Das and Siddhartha Ganguly are supported by the PMRF grant, from the Ministry of Human Resource Development, Govt. of India. Debasish Chatterjee acknowledges partial support of the SERB MATRICS grant MTR/2022/000656.}% <-this % stops a space
\thanks{$^{1}$The authors are with Systems and Control Engineering,
        Indian Institute of Technology Bombay, Powai, Maharashtra, India.
        Email: {\tt\small \{souvikd,sganguly,dchatter\}@iitb.ac.in.}}%
\thanks{$^{2}$The author is with Systems and Control Engineering,
        Indian Institute of Technology Bombay, Powai, Maharashtra, India.
        Email: {\tt\small anjali.m@sc.iitb.ac.in.}}%
}

\begin{document}

\maketitle
\thispagestyle{empty}
\pagestyle{empty}

\begin{abstract}
This article introduces a numerical algorithm that serves as a preliminary step toward solving continuous-time model predictive control (MPC) problems directly \emph{without} explicit time-discretization. The chief ingredients of the underlying optimal control problem (OCP) are a linear time-invariant system, quadratic instantaneous and terminal cost functions, and convex path constraints. The thrust of the method involves finitely parameterizing the admissible space of control trajectories and solving the OCP satisfying the given constraints at every time instant in a tractable manner without \emph{explicit} time-discretization. The ensuing OCP turns out to be a convex semi-infinite program (SIP), and some recently developed results are employed to obtain an optimal solution to this convex SIP. Numerical illustrations on some benchmark models are included to show the efficacy of the algorithm.
\end{abstract}


%%%%%%%%%%%%%%%%%%%%%%%%%%%%%%%%%%%%%%%%%%%%%%%%%%%%%%%%%%%%%%%%%%%%%%%%%%%%%%%%
\section{Introduction}
\label{s:intro}
%%%%%%%%%%%%%%%%%%%%%%%%%%%%%%%%%%%%%%%%%%%%%%%%%%%%%%%%%%%%%%%%%%%%%%%%%%%%%%%%

This article is concerned with an important component in the implementation of model predictive control (MPC) of continuous-time linear plants directly in continuous time. Of course, since digital implementation of controllers is normative today, there is an implicit discretization of time associated with the digital implementation, and this discretization due to digital implementation \emph{continues to stay} in our framework. Consequent to the preceding reconciliation, one issue must be immediately addressed in the context of a continuous-time implementation of MPC --- the time interval needed to compute solutions to the underlying finite-horizon constrained optimal control problem. This latency may be substantial due to the complexity of the optimal control problem, and is typically much longer than the interval of discretization associated with digital implementation; therefore, an \emph{online} implementation of this scheme may not be suitable.

To circumvent this difficulty, we establish a novel tractable numerical algorithm for numerically solving the finite-horizon constrained optimal control problem underlying an MPC strategy, \emph{directly in continuous-time}. This algorithm is fine-tuned to handle the most popular class of problem data with linear dynamics, quadratic cost functions, and affine constraints.
% and ensures robustness with respect to a natural class of bounded uncertainties.

With this numerical algorithm at hand, our goal, to be realized over our subsequent works, is to leverage the \emph{explicit MPC} technology: the implicit feedback in the MPC scheme is computed at a set of points on the admissible set of states, and an interpolation mechanism (guaranteeing tight control of the ensuing error) constructs an \emph{explicit} feedback by means of \emph{offline} numerical routines along the lines of \cite{ref:GanCha-22}. The result is a \emph{feedback} that is `close' to the implicit feedback map in the continuous-time MPC, and online implementation merely consists of evaluating this feedback map at each instant of time. Moreover, the extent of `closeness' of the interpolated feedback map and the implicit feedback of the MPC scheme is a choice of the designer. In particular, as will be demonstrated in subsequent reports, the ability to pre-assign an error margin lets us embed this uncertainty directly into the optimal control problem underlying the MPC strategy; no issues related to loss of feasibility arise as a consequence.

%% =============================================================================
\subsection*{Contributions}
%% =============================================================================

 It is evident from the preceding discussion that numerically computing solutions to the finite-horizon constrained optimal control/trajectory optimization problem is the critical issue in the context of explicit continuous-time MPC. Indeed, since an explicit map of the underlying feedback is our eventual target (and its synthesis strategy will be reported elsewhere), this offline computation at specific points in the feasible set constitutes the key step at hand, and we shall execute this step in this article. Here are our key contributions:
\begin{enumerate}[leftmargin=*]
    \item\label{contbu1} We introduce a new algorithm to solve a class of finite horizon continuous-time optimal control problems that serve as a first step towards addressing a continuous-time MPC  without \emph{explicit} time discretization. The established algorithm instead treats the admissible space of control functions corresponding to the underlying continuous-time OCP to be finite-dimensional (see Definition \ref{defn:discrete_admcon} in \S\ref{sec:ps} and the subsequent discussions).

    % \item \label{contbu2} Here we adopt a discretization scheme of the admissible space of control functions corresponding to the underlying continuous-time OCP which makes it finite-dimensional.
    \item\label{contbu3} In our approach the uncountable family of convex constraints in the continuous-time OCP is left intact, which makes the problem challenging but arguably more accurate in comparison to conventional methods such as \emph{direct multiple shooting}, \emph{direct collocation}, etc., that rely on time-discretization. We direct the readers to Remark \eqref{rem:Perspective} for a detailed discussion. 

	\item\label{contbu4} There have been previous efforts to address the finite-horizon optimal control problem in MPC directly in continuous-time; we draw attention to \cite{ref:PanRawMayMan-15, ref:cont_mpc_survey_magni} and the references therein, where a discretization scheme of the control input using piecewise linear functions is adopted \emph{on a finite time-grid} to ensure that the constraints are satisfied at the grid-points. In comparison, an internal feature of our algorithm intelligently selects a predefined number of time points so that the satisfaction of the constraints at those sample points \emph{implies} that the constraints hold for all time; certain convexity structures are central to this step that converts an infinitary condition into a finitary one. Moreover, the value function and the optimizers of the finite horizon optimal control problem (more precisely a reformulated version of the problem, see \eqref{e:fOCP}) is equal to the value function of the recast semi-infinite optimization problem; see \eqref{eq:g_func}, Theorem \ref{prop:prop}, and the discussion thereafter.
\end{enumerate}

\subsection*{Notation}
%% =============================================================================
We let \(\N \Let \aset{1,2,\ldots}\) denote the set of positive integers, \(\Nz\Let \N \cup \aset[]{0}\). The vector space \(\Rbb^{d}\) is equipped with standard inner product, \(\inprod{x}{y}\Let \sum_{j=1}^d x_j y_j\) for every \(x,y \in \Rbb^{d}.\) We denote the indicator function over a set \(S\) by \(\indic{S}(\cdot)\). The space of continuous functions on a domain \(X\subset \Rbb^{\nu_1}\) taking values in \(Y\subset \Rbb^{\nu_2}\) is denoted by \(\mathcal{C}\bigl(X;Y\bigr)\). 
 By \(\lpL[\infty]\bigl(X;Y\bigr)\) we denote the space of essentially bounded function on \(X\), taking values in \(Y\) with the essential supremum norm.

\section{Problem Statement}
\label{sec:ps}
Let us consider a linear time-invariant controlled dynamical system, modeled by the ordinary differential equation
\begin{equation} \label{eq:sys}
\dot \st(t) = A\st(t) + B\cont(t)\,\,\text{for a.e.\, } t \in [\tinit, \tfin],
\end{equation}
where \(\horizon>0\) is a fixed time horizon, \(\st(t) \in \Rbb^d\) is the vector of states and \(\cont(t) \in \Rbb^m\) is the vector of control/actions at time $t$, and \(A \in \Rbb^{d \times d}\), \(B \in \Rbb^{d \times m}\) are the state and the actuation matrices, respectively. Assume that the initial state \(x(0)\Let \param \in \Rbb^d\) and that the nonempty final set \(\finset \subset \Rbb^d\) are given, i.e.,
\begin{equation} \label{e:boundary constraint}
% x(0)=\param \in \Rbb^d\quad\text{and}\quad 
\st(\tfin) \in \finset \subset \Rbb^d,
\end{equation}
and that the control trajectory satisfies \(\cont(\cdot)\in \mathcal{U}\), where%\todo[inline]{Do we need to change the definition of \(\admcont\)?}
\begin{equation}
	\label{eq:control constraint}
	\hspace{-2mm}\admcon\Let \aset[\big]{u(\cdot) \in \lpL[\infty]([0,T];\admcont)\suchthat \cont(t) \in \admcont\,\,\text{for a.e}\,\,t \in \lcrc{0}{T}}.
\end{equation}
Here \(\admcont \Let \prod_{i=1}^m \admcont_i\), and the sets \(\admcont_i\), for each \(i=1,\ldots,m,\) are given nonempty, connected, and compact intervals with nonempty interiors. A control \(\cont(\cdot)\) is \textit{feasible} if it is measurable,\footnote{For us the word `measurability' always refers to Lebesgue measurability, and `a.e.'\ refers to almost everywhere relative to the Lebesgue measure.} satisfies the control constraint \eqref{eq:control constraint}, and the corresponding solution \(\st(\cdot)\) of \eqref{eq:sys} satisfies \eqref{e:boundary constraint}. We define the objective function by
\begin{equation}
    \label{eq:define_cost}
    V_{\horizon}(\param,\cont(\cdot))\Let \fcost\bigl(\st(\tfin)\bigr) + \int_{\tinit}^{\horizon} \rcost\bigl(\st(t), \cont(t)\bigr) \odif{t}.
\end{equation}
Over the feasible controls \(\cont(\cdot) \in \admcon\), the finite horizon continuous-time optimal control problem is typically posed:
\begin{equation}
	\label{e:OCP} \tag{\textsc{ocp}}
\begin{aligned}
& \inf_{\cont(\cdot)}	&& V_{\horizon}(\param,\cont(\cdot))\\
&  \sbjto		&&  \begin{cases}
\text{dynamics}\,\,\eqref{eq:sys}\,\text{and its associated data},\\
\st(\tinit)= \param,\,\, \st(\tfin) \in \finset \subset\Rbb^d,\\
\st(t) \in \admst \quad\text{for each}\,\, t\in [0,T], \\ 
\cont(\cdot) \in \admcon.
\end{cases}
\end{aligned}
\end{equation}

%%% on intractibility of the OCP %%%
\subsection*{Tractability of \eqref{e:OCP}}

Notice that the optimal control problem \eqref{e:OCP} is an \emph{infinite dimensional} optimization problem due to the fact that \(\cont(\cdot) \in \admcon\) and \(\admcon\) is infinite-dimensional. In contrast, \eqref{e:OCP} admits a finite-dimensional avatar when minimized over admissible control trajectories that are \emph{finitely} parameterized (e.g., if \(\admcon\) is the linear span of certain `basis functions'), which leads to a narrower set of admissible controls compared to the general setting of \eqref{e:OCP}; see \cite{MADCMN-2023}. This motivates the following definition: 

\begin{defn}\label{defn:discrete_admcon}
Let \(N \in \N\) and \(\dict \Let (\psi_i)_{i \in \Nz} \subset  \mathcal{C}([0,\horizon];\Rbb)\) be a dictionary of bounded continuous functions. We define the discretized set of admissible controls \(\admcon_{\dict}  \subset \mathcal{C}([0,\horizon];\Rbb)\) by 
% \[
% 	\admcon_{\dict} \Let  \linspan \aset[\big]{\lcrc{0}{\horizon}\ni s\mapsto \psi_i(s)\in\admact\suchthat \psi_i \in \dict \;\text{for }i = 1, \ldots, N},
% \]
\begin{align*}
    \admcon_{\dict} \Let \linspan \aset[\big]{ \psi_i:[0,\horizon] \ra \Rbb \suchthat \,\text{for }i = 1, \ldots, N},
\end{align*}
which is the span of finitely many basis functions from the dictionary \(\dict\).
\end{defn}
\begin{assum}[Standing Assumption]
\label{assump:standing}
We stipulate that an \(N\mbox{-}\)tuple of linearly independent elements from dictionary \(\dict\) have been extracted and fixed.
\end{assum}
For each \(i \in \aset[]{1,\ldots,m}\), the \(i\)-th component \(t \mapsto \cont_i(t)\) of the control trajectory \(t \mapsto \cont(t)\) is permitted to be a linear combination of basis functions \(t \mapsto \psi_j(t)\) for \(j=1,\ldots,N\). Define \(\Reg(t) \Let \bigl(\psi_1(t)\; \psi_2(t)\;\ldots \;\psi_N(t) \bigr)\). Then we have the component-wise parameterization: 
\begin{align} 
    \label{e:pcontrol}
    [0,\horizon] \ni t \mapsto \cont^{\dict}_i(t)= \sum_{j=1}^{N}\alpha_{i,j} \psi_j(t) \teL \inprod{\Param_i}{\Reg(t)} 
\end{align}
for \(i=1,\ldots,m\), where \(\alpha_{i,j} \in \Rbb\) are the \emph{control coefficients} (to be determined). Consequently, one can write the control trajectory \(t \mapsto \cont^{\dict}(t)\) in the compact form
\begin{align}
    \label{e:cont_param}
    \cont^{\dict}(t) = \begin{pmatrix}
        \Param_{1,1} & \Param_{1,2} & \cdots & \Param_{1,N}\\
        \Param_{2,1} & \Param_{2,2} & \cdots & \Param_{2,N}\\
        \vdots & \vdots & \ddots & \vdots\\
        \Param_{m,1} & \Param_{m,2} & \cdots & \Param_{m,N}
    \end{pmatrix} \Reg(t) \teL \Param \Reg(t),
\end{align}
for each \(t \in \lcrc{0}{\horizon}\), where \(\Param \in \Rbb^{m \times N}\) is the coefficient matrix. Note that with the above parameterization, the control trajectory \(t \mapsto \cont^{\dict}(t) \in \admcont\) and \(\cont^{\dict}(\cdot) \in \admcon\) --- this is because each component \(t \mapsto \cont_i^{\dict}(t) \in \admcont_i\) and \(\cont^{\dict}_i(\cdot) \in \mathcal{C}([0,\horizon];\Rbb)\) where \(i=1,\dots,m\). The choice of the generating functions \((\psi_j)_{j=1}^N\) depends primarily on the type of applications one has in mind and is up to the designer.

%%% FOCP %%%
The finite-dimensional optimal control problem over the set of feasible controls taking values from \(\admcon_{\dict}\), is given by:
\begin{equation}
	\label{e:fOCP} \tag{\textsc{focp}}
\begin{aligned}
& \min_{(\cont^{\dict}_i(\cdot))_{i=1}^m \subset \admcon_{\dict}}	&& V_T\bigl(\param,\cont^{\dict}(\cdot)\bigr) \\
&  \sbjto		&&  \begin{cases}
\text{dynamics}\,\,\eqref{eq:sys}\,\,\text{and its associated data,}\\
\st(\tinit)= \param,\,\, \st(\tfin) \in \finset \subset\Rbb^d,\\
\st(t) \in \admst \quad\text{for each}\,\, t\in [0,\horizon],\\
\cont^{\dict}(t) \in \admcont \quad\text{for each}\,\, t\in [0,\horizon], 
\end{cases}
\end{aligned}
\end{equation}
with the following data:
\begin{enumerate}[label=(\ref{e:fOCP}-\alph*), leftmargin=*, widest=b, align=left]
\item \label{OCPdata1} a quadratic \emph{instantaneous cost}
        \[
		     (\dummyx,\dummyu) \mapsto \rcost(\dummyx,\dummyu) \Let \inprod{\dummyx}{Q\dummyx}+ \inprod{\dummyu }{R\dummyu}\in \lcro{0}{+\infty}
		\]
		and a quadratic \emph{terminal cost}
		\[
			 \dummyx \mapsto \fcost(\dummyx)\Let \inprod{\dummyx}{P\dummyx}\in \lcro{0}{+\infty}
		\]
		with given positive semi-definite matrices \(Q = Q^{\top} \in \Rbb^{d \times d}\) and \(P=P^{\top} \in \Rbb^{d\times d}\), and a given positive definite matrix \(R = R^{\top} \in \Rbb^{m \times m}\).
  
\item\label{OCPdata2} The state constraint set \(\admst \subset \Rbb^d\) is a closed and convex set, and the terminal constraint set \(\finset \subset \admst \subset \Rbb^d\) is a compact and convex set. Each set \(\admst\) and \(\finset\) contains \(0 \in \Rbb^d\) in their respective interiors. 

%  \item \label{OCPdata4} 
% The admissible control trajectories in the continuous-time optimal control problem belongs to the admissible set of actions \(\admcon\).

% \item \label{OCPdata3}
% The map \(\fsblset \ni \dummyx \mapsto \valf(\dummyx) \in \lcro{0}{+\infty}\) is the \emph{value function} associated with the problem \eqref{e:OCP}, where \(\fsblset\) is the set of initial states for which the problem \eqref{e:OCP} is feasible. 
\end{enumerate}    

We denote the value of the objective in \eqref{e:fOCP} at initial state \(\param\) and admissible control \(\cont^{\dict}(\cdot)\) by \(V_{\horizon}(\param,\cont^{\dict}(\cdot))\), and 
\begin{equation}\label{eq:value funcion}
 \valf(\param) = \text{the optimal value of \eqref{e:fOCP}}; 
\end{equation}
since the initial states enter \eqref{e:fOCP} parametrically, of course \(\valf:\fsblset \lra \lcro{0}{+\infty}\) depends on the parameter \(\param\) where \(\fsblset\) denotes the set of initial states for which the problem \eqref{e:fOCP} is feasible. 
%%%%%%%%%%%Slater's condition%%%%%%%%%%%%%%%%%%
\begin{assum}
    \label{assum:slater_cond}
    We stipulate that Slater's condition holds for \eqref{e:fOCP}.\footnote{Assumption \ref{assum:slater_cond} states that the \eqref{eq:g_func} is strictly feasible and is a technical requirement for the proof of Theorem \ref{prop:prop} to go through.}
\end{assum}
%%% Prop: Existence of opt-st-con traj %%%
\begin{assum}\label{assmp:opt_traj_exst}We enforce the following assumptions:
\begin{itemize}[label = \(\circ\)]
    \item \label{it:as_it_1}In the context of the optimal control problem \eqref{e:fOCP} with its associated data \ref{OCPdata1}--\ref{OCPdata2}, we assume that \(\fsblset \neq \emptyset\), and that the problem \eqref{e:fOCP} admits an \emph{optimal state-action trajectory} \([0,\horizon] \ni t \mapsto \bigl(\st\opt(t;\param), \cont\opt(t;\param)\bigr)\);
    \item \label{it:as_it_2} the state-action trajectory \([0,\horizon] \ni t \mapsto \bigl(\st\opt(t;\param), \cont\opt(t;\param)\bigr)\) is a normal extremal in sense of \cite[Definition 2.1]{ref:vinter-galbraith-lipschitz}. 
\end{itemize}
\end{assum}
\begin{prop}\label{prop:opt_traj_exst}
Consider the optimal control problem \eqref{e:fOCP} with its associated data \ref{OCPdata1}--\ref{OCPdata2}. Let Assumption \eqref{assum:slater_cond} and Assumption \eqref{assmp:opt_traj_exst} hold, and let the constraint sets \(\admst, \finset\), and \(\admcont\) are polytopic. Then \([0,\horizon] \ni t \mapsto \cont\opt(t;\param)\) is Lipschitz continuous for every \(\param \in \admst_{\horizon}\).
\end{prop}
\begin{proof}
Notice that the assumptions \(\text{(H1)}\)--\(\text{(H4)}\) in \cite{ref:vinter-galbraith-lipschitz} are satisfied. The proof follows immediately by invoking \cite[Theorem 3.1]{ref:vinter-galbraith-lipschitz}.
\end{proof}
%%% Rmk: LipCon of OptCon %%%

% \begin{rem}\label{rmk:lip_con}
% \emph{(On the regularity properties of the optimal control)} A Lipschitz continuity property of the optimal control \([0,\horizon] \ni t \mapsto \cont\opt(t;\param)\) associated with the problem \eqref{e:fOCP} under the problem data \ref{OCPdata1}--\ref{OCPdata2}, normality assumptions and additional constraint qualification conditions (for example see the condition \(\mathsf{H5}\) in \cite{ref:depinho-2011lipschitz}) can be asserted. Furthermore, we emphasize that for the optimal control problem \eqref{e:OCP} under the same problem data \ref{OCPdata1}--\ref{OCPdata2}, the same existence result in Proposition \ref{prop:opt_traj_exst} and the Lipschitz regularity results hold. We refer the readers to \cite{ref:vinter-galbraith-lipschitz, ref:depinho-2011lipschitz} for more detail.
% \end{rem}

% \begin{enumerate}[label=(\ref{e:fOCP}-\alph*), leftmargin=*, widest=b, align=left]
% \item\label{FOCPdata} where the data \eqref{OCPdata1}--\eqref{OCPdata2} hold. 
% \end{enumerate}
Despite its finite-dimensional nature, the optimal control problem \eqref{e:fOCP} in a computation-theoretic sense is NP-hard in general, and involves uncountably many set-membership constraints since the state and the action variable in \eqref{e:fOCP} are subjected to the constraints \emph{at each instant} \(t \in [0,\horizon]\).

%% =============================================================================
\section{Solving \eqref{e:fOCP} via a convex SIP}
\label{sec:refSIP}
%% =============================================================================

The finite-dimensional problem \eqref{e:fOCP} is a convex program and admits a tractable approximation algorithm up to arbitrary precision, which is the subject of \S\ref{sec:refSIP}. To this end, we reformulate \eqref{e:fOCP} in the language of a convex semi-infinite program. Recall that the discretized control trajectory is given by the expression \(t \mapsto \cont^{\dict}(t) = \Param\Reg(t) \in \admcont\) as given in \eqref{e:cont_param}. 
Let us define the set of admissible \(\Param\)'s by 
\begin{align*}\adparam \Let \aset[\big]{\Param \suchthat \st(t) \in \admst, \;\Param \Reg(t) \in \admact \; \text{for each }t \in [0,T]}.\end{align*}
% To proceed further, observe that the discretized control trajectory \(\cont(\cdot)\) is given the finite expansion
% \begin{align}
%     \label{e:pcont_d}
%     t \mapsto \cont(t) &= \sum_{i=1}^N\alpha_i \psi_i(t)
%     =  \Reg(t){\transpose} \Param,
% \end{align}
% where \(\Reg(t) \Let (\reg_1(t),\reg_2(t), \ldots, \reg_N (t))\) for every \(t\in [0,\horizon]\); and the coefficient vector denoted by \(\Param \Let (\Param_1,\Param_2,\ldots,\Param_N)\) takes value in the set \(\adparam \Let \aset[\big]{(\Param_1,\ldots,\Param_N)\suchthat \sum_{i=1}^N \Reg(t){\transpose}\Param \in \admact \; \text{for each }t \in [0,T]} \subset \Rbb^N\). 
The solution of \eqref{eq:sys} starting from an initial state \(\param\) is given by the variation of constants formula:
\begin{align}\label{eq:p_sol}
t \mapsto \st(t) &= e^{At}\param + \int_{0}^{t} e^{A(t-\tau)} B \cont^{\dict}(\tau) \odif{\tau} \nn \\&
% &= e^{At}\st(0) + \int_{0}^{t} e^{A(t-\tau)} B \sum_{i=1}^d \alpha_i \psi_i(t) \odif{\tau}\\
% &= e^{At}\st(0) + \sum_{i=1}^d \alpha_i \int_{0}^{t} e^{A(t-\tau)} B \psi_i(t) \odif{\tau}\\
= e^{At}\param +  \int_{0}^{t} e^{A(t-\tau)} B \Param \Reg(\tau) \odif{\tau}.
\end{align}
Let us fix the notation 
\begin{align}\label{dvar}
\dvar \Let mN.
\end{align}
Then \eqref{e:fOCP} can be recast as a convex semi\(\mbox{-}\)infinite program: 
\begin{align}
    \label{e:r_SIP}
    &\min_{\Param \in \adparam}&& V_{\horizon}\bigl(\param,\Param \Reg(\cdot)\bigr) \nn\\
    &  \sbjto	&&  \begin{cases}
    %\text{dynamics}\,\,\eqref{eq:sys}\; \text{and its associated data},\\
    \st(\tinit)= \param,\,\, \st(\tfin) \in \finset \subset\Rbb^d,\\
    \st(t) \in \admst \quad\text{for each}\,\, t\in [0,\horizon],\\
    \Param \Reg(t) \in \admcont \quad\text{for each}\,\, t\in [0,\horizon],  
    \end{cases}
\end{align}
for each \(\param \in \fsblset\). Our main focus will now be on solving the problem \eqref{e:r_SIP}. To this end, define
the function \( [0,\horizon]^{\dvar} \ni (t_1,\ldots,t_{\dvar}) \teL \tseq \mapsto \gfunc(\tseq; \param) \in \Rbb\) for every \(\param \in \fsblset\), by
    \begin{align}\label{eq:g_func}
        &\gfunc(\tseq; \param) &&\Let \nn\\ &\min_{\Param \in \adparam}	&& V_{\horizon}\bigl(\param,\Param \Reg(\cdot)\bigr)\nn\\
        &  \sbjto		&&  \begin{cases}
        \st(\tinit)= \param,\,\, \st(\tfin) \in \finset \subset\Rbb^d,\\
        \st(t_i) \in \admst \quad\text{for all }\, i=1,\ldots,\dvar,\\
        \Param\Reg(t_i) \in \admcont \quad\text{for all }\, i=1,\ldots,\dvar.
        \end{cases}
\end{align}
 In the light of Assumption \ref{assmp:opt_traj_exst}, the `\(\min\)' in \eqref{e:r_SIP} and \eqref{eq:g_func} are well-defined, which leads to the following critical observation:
\begin{theorem}
    \label{prop:prop}
    Consider the continuous-time optimal control problem \eqref{e:fOCP} along with its associated data \ref{OCPdata1}-\ref{OCPdata2}. Consider also \eqref{e:r_SIP} and \eqref{eq:g_func} along with their associated data and notations. Suppose that Assumption \ref{assum:slater_cond} holds. If the finite sequence \(\tseq^{\ast} \Let (t^{\ast}_1,\ldots,t^{\ast}_{\dvar})\) is a solution of the maximization problem  
    \begin{align}\label{eq:sup_sip}
        \sup_{\tseq \in [0,T]^{\dvar}} \gfunc(\tseq;\param) \quad \text{for }\param \in \admst_{\horizon},
    \end{align}
    then \(\valf(\param) = \gfunc(\tseq^{\ast};\param),\) where \(\valf(\cdot)\) is defined in \eqref{eq:value funcion}.
\end{theorem}
\begin{rem}
    \label{rem:on_reformulation}
	Let us discuss the assertion of Theorem \ref{prop:prop}. The continuous-time optimal control problem \eqref{eq:g_func} has been formulated as a finite-dimensional optimization problem with a finite number of path constraints induced by the sequence \(\tseq\as\). The above proposition asserts that it is sufficient to consider only (but intelligently picked) \(\dvar\)-many discrete points \(\tseq\) at which the path constraints must be satisfied for the purpose of solving \eqref{e:r_SIP}. To recover such an \(\dvar\mbox{-}\)tuple of optimal point, the maximization problem \eqref{eq:sup_sip} in Theorem \ref{prop:prop} must be solved globally on \([0,\horizon]^{\dvar}.\)
\end{rem}
% \begin{rem}
%     Note that the dynamics from the set of constraints are removed by parametrizing the solution in terms of the decision variable. Moreover, the constraints on the control trajectory \(t \mapsto \cont^{\dict}(t)\) are translated to the parameters \(\Param \), which defines the set of admissible parameters taking values in the set \(\adparam\).
%     \end{rem}
\begin{proof}[Proof of Theorem \ref{prop:prop}]
Observe that the set \(\adparam\) is a compact and convex subset of \(\Rbb^{m \times N}\), in view of the fact that the sets \(\admst\) and \(\admcont\) are compact and convex. This implies that feasible set
\begin{align*}
        \mathcal{S} \Let \aset[\big]{\Param \in  \adparam \suchthat \text{the constraints } \text{in \eqref{e:r_SIP} hold}}
    \end{align*}
corresponding to \eqref{e:r_SIP} is closed and convex. Define the map 
\(
\Rbb^{m \times N} \ni \Param \mapsto F(\Param) \Let \Param \Reg(t) \in \admcont.
\)
Note that \(F(\cdot)\) is continuous in \(\Param\) and consequently any open set around the origin in \(\admcont\) has a preimage in \(\adparam\) which is open. Similar argument follows for the constraint \(\st(t) \in \admst\), implying that \(\adparam\) has a nonempty interior in the light of Assumption \ref{assum:slater_cond}. Moreover, the cumulative cost function 
\[
\adparam \ni \Param \mapsto  \int_{0}^{\horizon} \bigl(\inprod{\st(t)}{Q \st(t)}   + \inprod{\Param\Reg(t)}{R \Param\Reg(t)} \bigr)\odif{t},
\]
is both continuous and convex in \(\Param\). The existence of an optimizer \(\tseq^{\ast} \Let (t^{\ast}_1,t^{\ast}_2,\ldots,t^{\ast}_{\dvar}) \in [0,\horizon]^{\dvar}\) solving \eqref{eq:sup_sip} then follows from Assumption \ref{assum:slater_cond} and \cite[Theorem 4.1]{JMB:81}. Invoking \cite[Theorem 1]{ref:DasAraCheCha-22}, we assert that \(\valf(\param) = \gfunc(\tseq^{\ast};\param)\) for each \(\param \in \fsblset\).\end{proof}
\begin{rem}
%     The cost-per-stage function 
% \[
% \adparam \ni \Param \mapsto  \inprod{\st(t)}{Q \st(t)}   + \inprod{\Param}{\Reg{\transpose}(t)R \Reg(t) \Param},
% \]
% is a strictly convex function in \(\Param\). This is because \(R\) is a positive definite matrix which implies that the matrix \(\Reg{\transpose}(t)R \Reg(t)\) is also positive definite for every \(t\). Therefore, the Hessian of the map \(
% \adparam \ni \Param \mapsto  \inprod{\st(t)}{Q \st(t)}   + \inprod{\Param}{\Reg{\transpose}(t)R \Reg(t) \Param},
% \) is strictly greater than zero for every \(\Param\).
% Hence, by \cite[Proposition 2]{ref:DasAraCheCha-22}, 
Note that the solution to the maximization problem in Theorem \ref{prop:prop} is equal to the solution to the convex SIP \eqref{e:r_SIP} by a technique established in \cite{ref:ParCha-23}.
\end{rem}

\subsection*{Algorithm to solve \eqref{e:fOCP}} 
Here, we present an algorithm to solve the continuous-time optimal control problem directly. Recall that  
% the function \todo[fancyline]{Does this definition make sense before (8)?}
% \begin{align*}
%     &V_{\horizon}(\param;\Param\Reg(\cdot)) \\&\Let \int_{0}^{\horizon} \bigl( \inprod{\st(t)}{Q \st(t)}  + \inprod{\Param}{\Reg{\transpose}(t)R \Reg(t) \Param}\bigr)\odif{t} \\ & \qquad \qquad + \fcost\bigl(\st(\tfin)\bigr), 
% \end{align*}
the quantity 
\begin{align}\label{e:s}
    \Param \in \argmin_{\zeta \in \adparam}&\big\{ V_{\horizon}(\param;\zeta\Reg(\cdot))\,\big{|}  \,  \st(t_i) \in \admst,\, \zeta \Reg(t_i)   \in \admcont, \; \nn \\ & \hspace{10mm}\text{for all } i=1,\ldots,\dvar \big\} ,
\end{align}
where \(V_{\horizon}\bigl(\param;\Param \Reg(\cdot)\bigr)\) is defined in \eqref{eq:define_cost}.
%%% ALGO %%%
% \begin{algorithm}[!ht]
% \DontPrintSemicolon
% \SetKwInOut{ini}{Initialize}
% \SetKwInOut{giv}{Data}
% \giv{threshold number of iterations \(\tau\);}
% \ini{initialize constraint points: \(t_1^0, t_2^0, \dots, t_N^0 \in \lcrc{0}{1}\); initial guess for maximum value \(G_{\max}\)}; initial for the initial solution \(\Param\)

% \While{ \(n \leqslant \tau\)}
%     {
%         \emph{Sample} (via simulated annealing based algorithm) constraint time indices \((t_1^n, t_2^n, \dots, t_N^n) \in \lcrc{0}{1}^N\);\\
       
%         \emph{Evaluate} \(\gfunc_n \Let \gfunc(t_1,\ldots,t_N)\) as defined in \eqref{eq:g_func};\\
%         \emph{Recover} the solution \[\Param_n = \argmin_{\Param \in \adparam}\aset[\big]{ V_{\horizon}(\param;\Param)\suchthat  \st(t_i) \in \admst,\, \Reg{\transpose}(t) \Param  \in \admcont \;\text{for } t_1,t_2,\ldots,t_N \in [0,\horizon]},
%         \]
%         \If{\gfunc_n \geqslant \gfunc_{\max}}{\gfunc_{\max} \leftarrow \gfunc_n; \\
%             \Param \leftarrow \Param_n;\\}
%     }
% \caption{Randomized algorithm to solve a continuous-time optimal control problem}
% \label{alg:SIP_algo}
% \end{algorithm}

\begin{algorithm}[!ht]
	\SetAlgoLined
  	\DontPrintSemicolon
	\SetKwInOut{ini}{Initialize}
    \SetKwInOut{giv}{Data}
    \giv{threshold number of iterations \(\tau\);}
    \ini{initialize constraint points: \(t_1^0, t_2^0, \dots, t_{\dvar}^0 \in \lcrc{0}{T}\); initial guess for maximum value \(\gfunc_{\max}\); initial guess for the initial solution \(\ol{\Param}\);}
	\BlankLine
	\While{$n \leqslant \tau$}{
		\emph{Sample} (via simulated annealing-based algorithm) constraint time indices: \(\tseq^n =(t_1^n, t_2^n, \dots, t_{\dvar}^n) \in \lcrc{0}{T}^{\dvar}\) \;
		\emph{Evaluate} \(\gfunc_n \Let \gfunc(\tseq^n;\param)\) as defined in \eqref{eq:g_func} \;
		\emph{Recover} the solution  \(\Param_n\) as given in \eqref{e:s} \;
		\uIf{$\gfunc_n \geqslant \gfunc_{\max}$}{Set $\gfunc_{\max} \gets \gfunc_n$\;
        Set $\ol{\Param} \leftarrow \Param_n$}
        Update $n \gets n+1$ \;
	}
\caption{Simulated annealing based algorithm for continuous-time optimal control problem}\label{alg:sabaro}
\end{algorithm}

\begin{rem}
    \label{rem:Perspective}
    Contemporary methods for direct trajectory optimization involve discretization of both time and the space of admissible control functions via employing some form of approximation at the level of control \cite{ref:diehl2011numerical, ref:SG:NR:DC:RB-22} or both states and control \cite{ref:garg-2010}. These methods are broadly classifiable into two categories --- \emph{direct multiple shooting} and \emph{direct collocation} --- that involve discretization of the time axis, and discretization of the dynamics by means of suitable numerical integrator schemes, adding finitely many equality constraints to the nonlinear program, one at each point of discretization. These schemes produce a finitely parametrized family of trajectories that are then subjected to finitely many set-membership constraints, one each at the points of discretization: the uncountable family of path constraints are relaxed to finitely many constraints on the discretized time grid. The resulting finite-dimensional nonlinear program is then solved. On the other hand, the uncountable family of constraints in \eqref{e:OCP} is left intact in our approach, which makes it far more challenging but consequently and arguably, more accurate compared to conventional direct methods.
\end{rem}

%% =============================================================================
\section{Stability}
\label{sec:main_result}
%% =============================================================================
Recall from \S\ref{s:intro} that we seek to leverage the explicit MPC technology along the lines of \cite{ref:GanCha-22}. To this end we include a technical result concerning the stability of the dynamical system \eqref{eq:sys} under the feedback policy \(\dummyx \mapsto u\opt^{\mathrm{o}}(\dummyx) \Let u\opt(0;\dummyx)\). 
% To this end, we show a Lyapunov descent-like property of the value function \(\valf(\cdot)\). 
\begin{assum}\label{assump:stability_assumptions}
We stipulate that \(\fcost(\cdot), \finset\), and \(\rcost(\cdot,\cdot)\) must satisfy following properties \cite{ref:MayFal-19}:  
\begin{enumerate}[label=(\ref{assump:stability_assumptions}-\alph*), leftmargin=*]
%%% Assmp 1 %%%
        % \item \label{eq:stability_prop1} Let the data \eqref{OCPdata2} hold, and there exist class \(\mathcal{K}\) functions \(\alpha_1(\cdot), \alpha_2(\cdot)\) such that
        %  \begin{itemize}[label=\(\triangleright\), leftmargin=*]
        %     \item \(\rcost(\dummyx,\dummyu) \ge \alpha_1 \bigl(|\dummyx|\bigr)\) for every \(\dummyx \in \fsblset\) and for every \(\dummyu \in \admcont\), 
        %     \item \(\fcost(\dummyx) \le \alpha_2 \bigl(|\dummyx|\bigr) \,\,\) for every \(\dummyx \in \finset\).
        % \end{itemize}
%%% Assmp 2 %%%        
	\item\label{eq:stability_prop2} There exists a feedback given by \(\finset\ni \dummyx\mapsto g_F (\dummyx)  \in \admcont\), such that the terminal set \(\finset\) is positively invariant for the system \(\dot{\dummyx}=A\dummyx+B g_F(\dummyx)\).
%%% Assmp 2 %%%        
        \item \label{eq:stability_prop3} 
         The following inequality holds for all \(\dummyx \in \finset\):  
\begin{equation}\label{e:CLF}
\hspace{-2mm}\inprod{\nabla_{\dummyx}\fcost(\dummyx)}{A\dummyx+B g_F(\dummyx)} \le -\rcost(\dummyx,g_F(\dummyx)).
\end{equation}\
     \end{enumerate}
    \end{assum}
Here is our key technical result:
%%% main theorem %%%
\begin{theorem}\label{thrm:stability}
	Consider the constrained optimal control problem \eqref{e:fOCP} along with its data \ref{OCPdata1}--\ref{OCPdata2}, and suppose that Assumption \eqref{assump:stability_assumptions} holds. Let \(\fsblset \neq \emptyset\), let \( [0,T] \ni t \mapsto \bigl(\st\opt(t;\param),\cont\opt(t;\param)\bigr)\) be an optimal state-action trajectory, and define \( \fsblset \ni \dummyx \mapsto \cont\opt^{\mathrm{o}}(\dummyx) \Let \cont\opt(0;\dummyx)\). Let \(V_{\horizon}\as(\cdot)\) be differentiable, then we have the following inequality for all \(\dummyx \in \fsblset\):
\begin{align}\label{value_func_descend}
\inprod{\nabla_{\dummyx}V_T\as(\dummyx)}{A\dummyx+B\cont\opt^{\mathrm{o}}(\dummyx)} \le -\rcost(\dummyx,\cont\opt^{\mathrm{o}}(\dummyx)).
\end{align}
\end{theorem}
\begin{rem}
\label{rem:on_stability}
Theorem \eqref{thrm:stability} demonstrates that \emph{under the feedback policy} \( \fsblset \ni \dummyx \mapsto \cont\opt^{\mathrm{o}}(\dummyx) \Let \cont\opt(0;\dummyx)\) the closed-loop system \(\dot{\st}(t) = A\st(t)+ B \cont\opt^{\mathrm{o}}(\st)\) is stable via the descent-like property \eqref{value_func_descend} of the value function \(\valf(\cdot)\) along the closed-loop trajectories of the system \eqref{eq:sys}.
\end{rem}

\begin{rem}[On differentiability of the value function]
Notice that in Theorem \eqref{thrm:stability} we assumed that \(V_{\horizon}\as(\cdot)\) is differentiable. One instance of such cases arises when the constraint sets \(\admst, \finset\), and \(\admcont\) are polytopic in nature, which is quite common in MPC literature \cite{ref:BBM_book}.  Differentiability of \(V_{\horizon}\as(\cdot)\) then follows immediately from \cite[Theorem 7.3]{ref:param_opt_still}.
\end{rem}
\begin{proof}[Proof of Theorem \ref{thrm:stability}]
Let \(\finset \ni \dummyx \mapsto g_F(\dummyx) \in \admcont\) be a map for which \(\fcost(\cdot)\) satisfy condition \eqref{e:CLF} in Assumption \eqref{assump:stability_assumptions}. For each \(\param \in \fsblset\), consider the optimal state-action trajectory \([0,\horizon] \ni t \mapsto \bigl(\st\opt(t;\param),\cont\opt(t;\param)) \in \admst \times \admcont\). Fix \(\stepsz>0\). Let us recall the policy \(g_F(\cdot)\) from Assumption \ref{assump:stability_assumptions} and consider the closed-loop system under \(g_F(\cdot)\):
\begin{align}
   \dot{z}(t) = A {z}(t) + B g_F(z(t)) \quad \text{with  } z(\horizon) = \st\as(\horizon;\dummyx)
\end{align}
for \(t \in \lcrc{\horizon}{\horizon+h}\). We consider the concatenated trajectory \(\lcrc{\stepsz}{\horizon+\stepsz} \ni t \mapsto \bigl(\widehat{\st}(t),\widehat{\cont}(t)\bigr)\) defined by
\begin{align}\label{e:cont_trajectories} 
        &t \mapsto \widehat{\st}(t) \Let \st\opt(t;\param) \indic{\lcrc{\stepsz}{\horizon}}(t) + z(t;\st\opt(\horizon;\param)) \indic{\lorc{\horizon}{\horizon+\stepsz}}(t),\nn \\  
        &t \mapsto \widehat{\cont}(t) \Let \cont\opt(t;\param) \indic{\lcro{\stepsz}{\horizon}}(t) + g_F(t;\st\opt(\horizon;\param)) \indic{\lcrc{\horizon}{\horizon+\stepsz}}(t). \nn
\end{align}
From \eqref{eq:define_cost}, the cost function is written as
\begin{equation}\label{valf_ineq_1}
   \begin{aligned}
   &V_{\horizon} \bigl( \st\opt(h;\param), \widehat{\cont}(\cdot) \bigr) =
   \\& \int_{h}^{\horizon+h} \rcost \bigl(\widehat{\st}(s),\widehat{\cont}(s) \bigr) \,\odif{s} + \fcost( \widehat{\st}(\horizon+h)) \nn \\
   & = \fcost( \widehat{\st}(\horizon+h)) - \fcost(\st\opt(\horizon;\param))+ \valf(\param)\nn \\ & + \int_{\horizon}^{\horizon+h}\rcost(\widehat{\st}(s),\widehat{\cont}(s))\,\odif{s})  - \int_0^h \rcost (\st\opt(s;\param),\cont\opt(s;\param))\,\odif{s}. 
    \end{aligned}
\end{equation}
% This implies that 
% \begin{multline*}
%    V_{\horizon} \bigl( \st\opt(h;\param), \hat{\cont}(\cdot) \bigr) -\valf(\param) \\\le  - \int_0^h \rcost (\st\opt(s;\param),\cont\opt(s;\param))\,\odif{s}, 
% \end{multline*}  
% where the above inequality follows from  the descent-like property \eqref{e:CLF} in Assumption \ref{assump:stability_assumptions}.
Notice that, by suboptimality of the trajectory \(t \mapsto \widehat{\cont}(t)\), the value function satisfies \(\valf(\st\opt(h;\param)) \leqslant V_{\horizon} \bigl( \st\opt(h;\param), \widehat{\cont}(\cdot) \bigr)\) and we write the above equality as 
\begin{align}\label{valf_ineq_2}
   &\valf(\st\opt(h;\param)) - \valf(\param)\nn\\
   & \le \fcost( \widehat{\st}(\horizon+h)) - \fcost(\st\opt(\horizon;\param)) \nn\\& \int_{\horizon}^{\horizon+h}\rcost(\widehat{\st}(s),\widehat{\cont}(s))\,\odif{s}) -\int_{0}^{h} \rcost (\st\opt(s;\param),\cont\opt(s;\param))\,\odif{s}.  
\end{align}
Dividing both sides of the above inequality by \(h\), taking \(\limsup_{h\downarrow 0}\) on both sides, and invoking inequality \eqref{e:CLF} in Assumption \eqref{assump:stability_assumptions} we get the desired result \eqref{value_func_descend} for each \(\dummyx \in \fsblset\). The proof is complete.
\end{proof}
% \begin{rem}
% The stability assumption \eqref{assump:stability_assumptions} indicates that the OCP \eqref{e:fOCP} is feasible under the sub-optimal control \(\dummyx \mapsto \widehat{\cont}(\dummyx)\)
%    The recursive feasibility pr 
% \end{rem}

\begin{rem}
    It is clear from \ref{eq:stability_prop2}-\ref{eq:stability_prop3} in Assumption \eqref{assump:stability_assumptions} and from the inequality \eqref{valf_ineq_2} that \(\widehat{u}(\cdot)\) is feasible for the OCP \eqref{e:fOCP} over the horizon \(\stepsz\) to \(T+\stepsz\), which asserts successive feasibility of OCP \eqref{e:fOCP}.
\end{rem}

%\input{proof}
%% =============================================================================
\section{Discussion and numerical experiment}
\label{sec:num_exp}
%% =============================================================================

In this section, with the aid of two numerical experiments, we demonstrate the ability of our algorithm to generate constrained state and control trajectories that satisfy the given constraints for \emph{all} time. We again emphasize the fact that unlike conventional methods we directly deal with uncountably many path constraints without any time discretization. The procedure and the tractable algorithm have been discussed in detail above; see \S\ref{sec:refSIP}. We also highlight the presence of a \emph{turnpike behavior} \cite{ref:LG_turnpike} in our second example in \S\ref{subsec:IP_main}. This phenomenon serves as a motivation to employ model predictive control strategies, which we will report in our subsequent articles.  We report that all numerical experiments are conducted with Python \(3.8.18\) (SciPy \(1.10.1\)) in a computer with \(11\)th Gen Intel(R) Core(TM) \(i7\mbox{-}11370\)H CPU clocked at \(3.30\)GHz with \(16\)GB RAM. We compared all the numerical results obtained by employing our method with \(\quito\) which is a direct multiple shooting technique \cite{ref:SG:NR:RA:MS:DC}.\footnote{On MATLAB-19, in a computer with \(8\)th Gen Intel(R) Core(TM) \(i5-8250\) CPU clocked at \(1.60\)GHz with \(8\)GB RAM.} Notice that the difference in computation time is natural due to the presence of a global optimization step which is interfaced off-the-shelf without any fine-tuning; that said we emphasize the fact that online optimization time statistics are irrelevant in the context of explicit MPC, which is our final goal. 

%%% Bryson Denham problem %%%
\subsection{The Bryson-Denham problem}
\label{subsec:BD_main}

We begin with the second-order state-constrained  benchmark Bryson-Denham system \cite{ref:Bry75}. The optimal control problem is given by:
\begin{equation}
\label{eq:Bryson_OCP}
\begin{aligned}
	& \min_{\cont(\cdot)}	&&  J\bigl(\cont(\cdot)\bigr) \Let \int_{0}^{1} \cont(t)^2 \odif{t} \\
&  \sbjto		&&  \begin{cases}
\dot{\st}_1(t)=\st_2(t), \,\dot{\st}_2(t)= \cont(t),\\
% \dot{x}_1(t)=x_2(t),\,\,\dot{x}_2(t)=u(t),\\
% \st(0)=(0,1)\,\,\text{and}\,\,
% \st(1)=(0,-1),\\
\st_1(0)=0 \,\,\text{and}\,\,\st_2(0)=1,\\
\st_1(1)=0 \,\,\text{and}\,\,\st_2(1)=-1,\\
\st_1(t)\le l \Let \frac{1}{9} \,\,\text{for all}\,\,t \in [0,1].
\end{cases}
\end{aligned}
\end{equation}
The analytical solution of the problem \eqref{eq:Bryson_OCP} can be found in \cite{ref:Bry75}. The problem is challenging from a numerical viewpoint because of the presence of the path constraints \(\st_1(t) \le 1/9\) for all time \(t \in [0,1]\). The implementation details are presented in \S\ref{subsubsec:BD_implementation}. 

\subsubsection{Numerical simulation}
\label{subsubsec:BD_implementation} 
We consider the sinusoidal basis functions for the discretized control trajectory in \eqref{e:pcontrol}:
\begin{equation}
\label{eq:sine_basis}
\begin{aligned}
    \psi^\mathbf{s}_i(t) \Let \begin{cases}
        1 & \text{if}\quad i=1,\\
        \sin(2\pi(i-1)t) & \text{if}\quad i \in \{2,\dots, \tfrac{N+1}{2}\},\\
        \cos\Big(2\pi \Big(i-\frac{N+1}{2}\Big)t\Big) & \text{if}\quad i \in \{\tfrac{N+3}{2},\dots,N\}.
    \end{cases}
\end{aligned}
\end{equation}
Substituting the expression \eqref{e:pcontrol} with \(\psi_i(\cdot)=\psi^{\mathbf{s}}_i(\cdot)\) in the cost function of \eqref{eq:Bryson_OCP} leads to:
\begin{equation}\label{e:qr cost func}
    J(\alpha) = \alpha\transpose P \alpha \quad\text{with}\,\,\alpha = \pmat{\alpha_1 & \alpha_2 & \cdots &\alpha_N}\transpose, \nn
\end{equation}
and the matrix \(P\) is defined by
\begin{align}\label{e:grammian}
	P & \Let  \pmat{
		\langle \psi_1, \psi_1 \rangle & \langle \psi_2, \psi_1 \rangle & \cdots & \langle \psi_N, \psi_1 \rangle\\
		\vdots & \vdots & \ddots & \vdots\\
		\langle \psi_1, \psi_N \rangle & \langle \psi_2, \psi_N \rangle & \cdots & \langle \psi_N, \psi_N \rangle
        }, \nn 
\end{align}
where \(\inprod{\psi_i(\cdot)}{\psi_j(\cdot)}\) is the standard \(\horizon\)-horizon \(\lpL[2]\)-inner product. Now we solve the global maximization problem:
    \begin{align}
        \sup_{\tseq \in [0,\horizon]^{\dvar}} \gfunc(\tseq;\param) \;\; \text{where }x \Let \bigl(
        \st_1(0)\,\,\st_2(0)\bigr),\nn
    \end{align}
and the objective \(\gfunc(\cdot;x)\) for each \(x\in \fsblset\) is defined in the following fashion:

\begin{equation}
\begin{aligned}
    &\gfunc(\tseq;\param) \Let \nn \\ 
    & \min_{\alpha}	&& \hspace{-5mm}\alpha\transpose P \alpha  \\
    & \sbjto		&& \hspace{-5mm} \begin{cases}
    %\dot{\st}_1(t)=\st_2(t), \,\dot{\st}_2(t)= \cont(t),\\
    % \dot{x}_1(t)=x_2(t),\,\,\dot{x}_2(t)=u(t),\\
    % \st(0)=(0,1)\,\,\text{and}\,\,
    % \st(1)=(0,-1),\\
    %\st_1(0)=0 \,\,\text{and}\,\,\st_2(0)=1,\\
    \Param \in \adparam,\, x_1(0) = 0 \text{ and }x_2(0)=1,\\
    \st_1(1)=0 \,\,\text{and}\,\,\st_2(1)=-1,\\
    \st_1(t_i)\le l \Let \frac{1}{9}\; \text{for } t_1,\ldots,t_{\dvar} \in [0,\horizon].
    \end{cases}
\end{aligned}
\end{equation}
The finite-dimensional inner minimization problem is solved using the `SLSQP' solver in Python (see \cite[Chapter 18]{ref:JNSJ-06} for the algorithm) and the outer maximization problem is solved using simulated annealing \cite[Chapter 13]{ref:OH-02} with different number of basis functions for control trajectory parameterization. The state and the control trajectories are depicted in Fig. \ref{fig:BD_states_51_basis} and Fig. \ref{fig:BD_control}, it can be seen that the path constraint on \(\st_1(\cdot)\) is satisfied for all \(t \in [0,1]\) with \(51\) basis functions for control parameterization. The evolution of the cost function with an increasing number of iterations is shown in Fig. \ref{fig:BD_cost}. The analytical value of the cost function for this OCP is \(8\), and it can be clearly seen that with the established method while the cost saturates to a sub-optimal value with \(11\) basis functions, it quickly converges to the optimal value with \(31\) and \(51\) basis functions. The CPU time with \(51\) basis functions is \(22.07964\) sec for \(550\) iterations of the global optimization; see \ref{tab:opt_data_BD} for more details.
\begin{figure}[!ht]
  \centering
  \begin{subfigure}[b]{0.49\linewidth}
    \includegraphics[width=\linewidth]{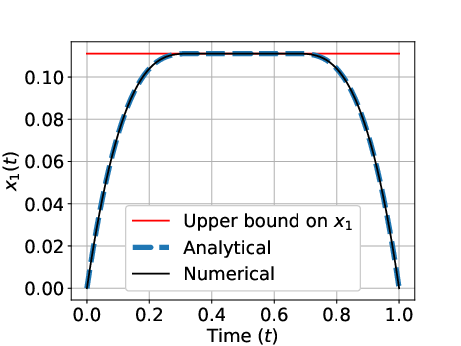}
  \end{subfigure}
  \begin{subfigure}[b]{0.49\linewidth}
    \includegraphics[width=\linewidth]{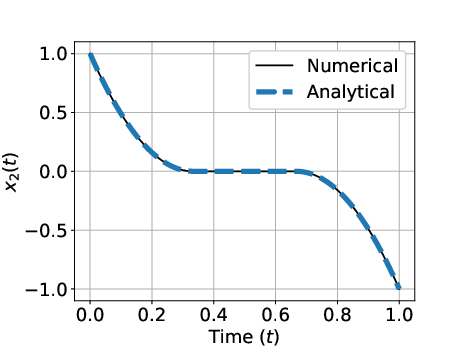}
  \end{subfigure}
 \caption{Time evolution of the state \(\st_1(\cdot)\) and \(\st_2(\cdot)\) with \(51\) basis functions employed in \eqref{e:pcontrol}. It can be seen that the path constraint on the state \(\st_1(\cdot)\) is respected. }
 \label{fig:BD_states_51_basis}
\end{figure}

\begin{figure}[tbp]
\centerline{\includegraphics[scale=0.4]{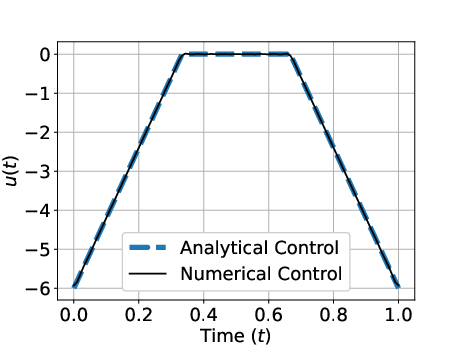}}
\caption{Time evolution of the control \(\cont(\cdot)\) with \(51\) basis functions employed in \eqref{e:pcontrol}. The analytical solution is also plotted in the same time horizon; the closeness of the analytical and the solution obtained from our algorithm is noteworthy.}
\label{fig:BD_control}
\end{figure}

%%% Cost value with 11, 31 and 51 basis %%%
\begin{figure}
    \includegraphics[draft=false, width=\linewidth]{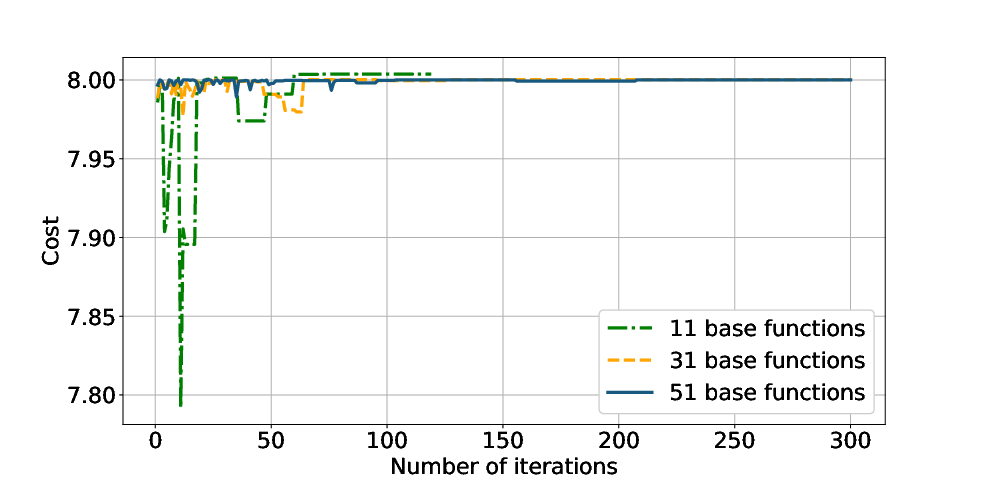} 
    \caption{Convergence of the cost value using \(11\), \(31\), and \(51\) basis functions with simulated annealing.}
    \label{fig:BD_cost}
\end{figure}

\begin{table}[htbp]
\caption{A comparision of optimization statistics}
\begin{tblr}{l|c|c}
	\hline[2pt]
		\SetRow{azure9}\SetCell[c=4]{c} Solver specifications and computation time
& 3-2
& 3-3
& 3-4
\\
\hline[1pt]
Method & Solver specs & Computation time (secs) \\ 
 \hline
Ours & Dual annealing  & 22.07964 \\ \hline
\(\quito\) & quadprog & 4 (with Gaussian kernel, step-size \\ & &  \(h = 0.02\), \(\mathcal{D}=2\))\\
\hline[2pt]
\end{tblr}
\centering
\label{tab:opt_data_BD}
\end{table}

\subsection{Stabilization of an inverted pendulum on a cart}
\label{subsec:IP_main}

In our second numerical experiment, we consider an inverted pendulum on a movable cart with a restricted cart length as shown in Fig. \ref{fig:inverted pendulum}. This is a benchmark under-actuated system with the horizontal force acting on the cart as the only control variable; see \cite[\S 4.3]{ref:DCAPHKJ-02} for more details. The inverted pendulum is in an upright position pivoted to the cart having an unstable equilibrium point. Our objective is to stabilize the inverted pendulum on the cart for a given time \(\horizon > 0\). The cart-pendulum system is a nonlinear time-invariant fourth-order dynamical system and we consider a linear approximation around the vertical position of the inverted pendulum.

\subsection*{Linear system formulation:}\label{subsubsec:IP_implementation}
\noindent We linearize the dynamics of the inverted pendulum system around the origin with the states \(\st_1(t) \Let \st(t)\) as the position of the cart, \(\st_2(t) \Let \theta(t)\) as the angle made by the pendulum with the vertical, \(\st_3(t) \Let \odv{\st_1}{t}(t)\) as the velocity of the cart, \(\st_4(t) \Let \odv{\st_2}{t}(t)\) as the angular velocity of the pendulum and the  control input \(t \mapsto \cont(t)\) as the horizontal force acting on the cart. 
% \begin{equation}\label{e:inverted pendulum}
%     \dot \st(t) = A\st(t) + B\cont(t),
% \end{equation}
% where \( \st \Let \bigl( \st_1 \,\, \st_2 \,\, \st_3\,\,  \st_4\bigr)\transpose\) and 
The system matrix \(A\) and the input matrix \(B\) are defined by
\begin{equation*}
A \Let \begin{bmatrix}
	0 & 0 & 1 & 0 \\
	0 & 0 & 0 & 1 \\
	0 & \frac{-m_0^2gl^2}{(J_0+m_0l^2)p} & 0 & 0 \\
	0 & \frac{m_0gl(m_0+M)}{(J_0+m_0l^2)p} & 0 & 0	
    \end{bmatrix},\,
B \Let  \begin{bmatrix}
	0 \\
	0 \\
	\frac{1}{p} \\
	\frac{-m_0l}{(J_0+m_0l^2)p}
    \end{bmatrix},
\end{equation*}
where \(p \Let m_0+M-(m_0^2l^2/J_0+m_0l^2)\).
Here, \(m_0\) is the mass of the pendulum, \(M\) is the mass of the cart, \(2l\) is the length of the pendulum,  \(L\) is the length of the cart track available on both sides of the origin, and \(J_0\) is the moment of inertia of the pendulum (see \cite{ref:DCAPHKJ-02} for more details on the model).
% To justify the linear approximation we constraint the states to be within an \(\varepsilon\)-ball around the equilibrium position at each instant of time \(t \in \lcrc{0}{\horizon}\) where the equilibrium position is \(\st_{\mathrm{eq}} \Let \bigl(0\,\, 0\,\, 0\,\, 0\bigr)\transpose\) and the constraints are
% \begin{equation}
% \abs{\st_1(t)} \leq \varepsilon_1, \,\abs{\st_2(t)} \leq \varepsilon_2, \,\abs{\st_3(t)} \leq \varepsilon_3, \, \abs{\st_4(t)} \leq \varepsilon_4, \nn
% \end{equation}
% for every \(t \in [0,\horizon]\).
Let us consider the quadratic cost  
\begin{equation}
J\bigl(u(\cdot)\bigr) \Let \int_0^T u(t)^2 \odif{t}, \nn
\end{equation}
with horizon \(\horizon = 10\) secs. From the above data, via the parameterized control action in \eqref{e:pcontrol} we can formulate a convex SIP as
\begin{equation}
	\label{e:pendulum_SIP}
    \begin{aligned}
		& \min_{\alpha \in \Rbb^{m \times N}} &&\alpha\transpose P \alpha\\
		& \sbjto &&
		\begin{cases}
			\alpha \in \adparam, & \\
			\abs{\st(t_i)} \leq \varepsilon \,\,\text{for } t_1,\ldots,t_{\dvar} \in [0,\horizon],\\ \text{where} \,\,\varepsilon \Let (\varepsilon_1 \,\, \varepsilon_2 \,\, \varepsilon_3\,\,\varepsilon_4)\transpose	.
		\end{cases}
    \end{aligned}
\end{equation}
The corresponding maxmin problem is given by:
\[
\max_{\tseq \in \lcrc{0}{\horizon}^{\dvar}} \gfunc(\tseq;\param),
\]
where \(\param \Let \st(0)\) and the objective function \(\gfunc(\cdot;\param)\) for each \(\param \in \fsblset\), is defined in \eqref{e:pendulum_SIP}.
% \begin{equation}
% \gfunc(\tseq;\param) = \min_{\alpha \in \adparam} \aset[\big]{\alpha\transpose P \alpha \suchthat \abs{\st(t_i)} \leq \varepsilon,\, \text{for }\,i = 1,\dots,\dvar}. \nn
% \end{equation}
Here, the inner minimization problem is constrained on a finite set of time instants \((t_1,\ldots,t_{\dvar})\).

\begin{figure}
\centering
\includegraphics[width=0.5\linewidth]{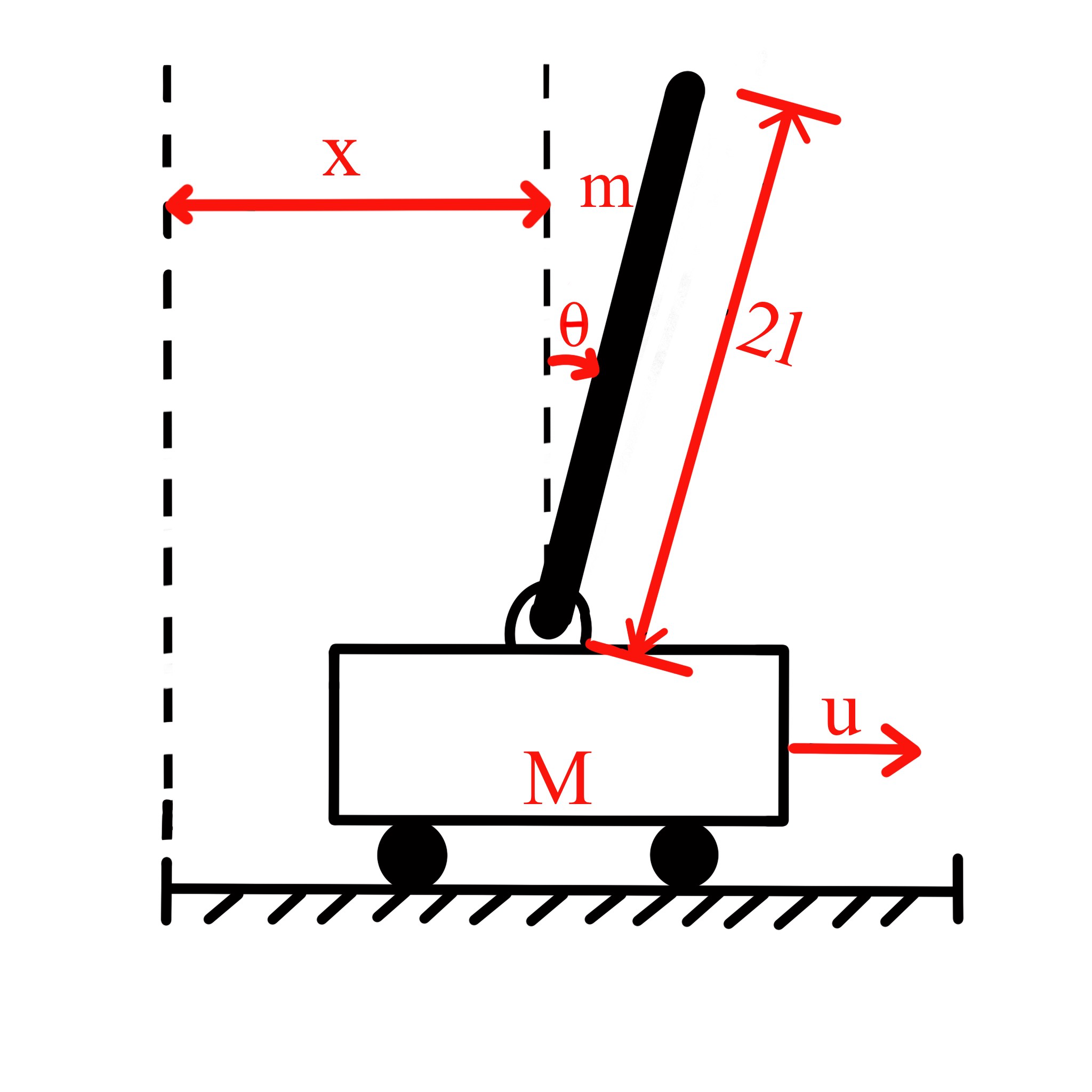}
\caption{Inverted pendulum on a cart setup.}
\label{fig:inverted pendulum}
\end{figure}

\subsubsection{Numerical simulation}
In the numerical simulation, we consider the following data:
\( m_0 = 0.2\, \mathrm{kg}, \; M = 3\, \mathrm{kg}, \; l = 1.5\, \mathrm{m}, \; L = 0.5\,\mathrm{m}.\) The initial position of the cart is at the origin and the track length is up to \(0.5\,\mathrm{m}\) on either side of the origin. The initial position of the states is \(\st(0) = (0, 0.035, 0, 0)\transpose\) where the angle made by the pendulum with the vertical is two degrees to the right. The system is simulated for \(\horizon = 10\) sec. In this experiment, the number of basis functions is considered to be \(N = 31\) and the bounds on the states are \(\varepsilon_1 = 0.5\), \(\varepsilon_2 = 0.07\), \(\varepsilon_3 = 0.5\) and \(\varepsilon_4 = 0.1\). We solve the ensuing optimization problem using Algorithm \ref{alg:sabaro}, employing the simulated annealing routine for solving the global optimization, and the CPU time with \(31\) basis functions was \(54.01169\) secs for \(900\) iterations of the global optimization; see \ref{tab:opt_data_IP} for more details. This is a practical number given the length of time horizon \(\horizon = 10\) secs. The state trajectories are depicted in Fig. \ref{fig:IP_pos_vel}  and Fig. \ref{fig:IP_ang_angvel}; we can see that all the states reach the equilibrium position and stay within the specified bounds. 
\par In Fig. \ref{fig:IP_turnpike}, we observe hints of the turnpike property in the near-optimal \(\st_2(\cdot)\) trajectory which is the angle made by the inverted pendulum with the vertical axis from an initial state as it reaches the steady state and departs close to the bound at the end of the time interval \(\lcrc{0}{10}\). This observation points strongly to the justification of employing MPC in continuous-time.
% \begin{figure}
% \centering
% \includegraphics[width=8.7cm]{Plots/Position_10sec_1.png}
% \caption{Position of the cart with 11 basis functions in the control action from time \(0\) to \(10\) sec.}
% \label{fig:position}
% \end{figure}

% \begin{figure}
% \centering
% \includegraphics[width=8.7cm]{Plots/Angle_10sec_1.png}
% \caption{Angle made by the inverted pendulum with the vertical with 11 basis functions in the control action from time \(0\) to \(10\) sec.}
% \label{fig:angle}
% \end{figure}

% \begin{figure}
% \centering
% \includegraphics[width=8.7cm]{Plots/Velocity_10sec_1.png}
% \caption{Velocity of the cart 11 basis functions in the control action from time \(0\) to \(10\) sec.}
% \label{fig:velocity}
% \end{figure}

% \begin{figure}
% \centering
% \includegraphics[width=8.7cm]{Plots/Angular_velocity_10sec_1.png}
% \caption{Angular velocity of the cart 11 basis functions in the control action from time \(0\) to \(10\) sec.}
% \label{fig:ang_vel}
% \end{figure}

\begin{figure}[!ht]
\centering
  \includegraphics[width=\linewidth]{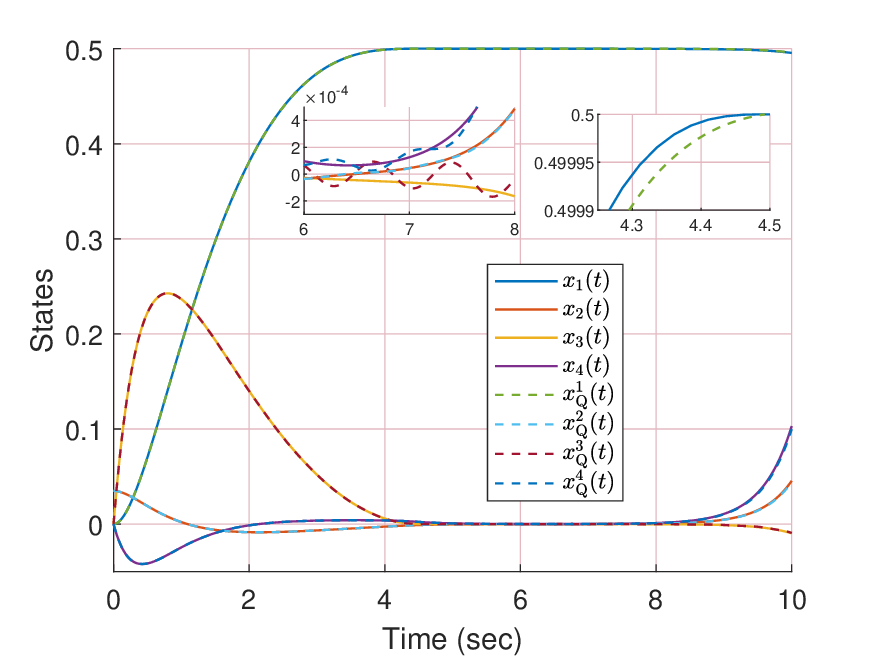}
 \caption{State trajectories \(\st_i(\cdot)\) and \(\st_{\mathrm{Q}}^i(\cdot)\) for \(i=1,\ldots,4\) when computed using our algorithm with \(31\) basis functions and when computed with the direct transcription software \(\quito\). It can be seen that the path constraints on the states are respected and the trajectories generated by our algorithm and \(\quito\) matches very closely.}
 \label{fig:IP_pos_vel}
\end{figure}
% \begin{figure}[!ht]
%   \centering
%   \begin{subfigure}[b]{0.49\linewidth}
%     \includegraphics[width=\linewidth]{Plots/pos_51.png}
%   \end{subfigure}
%   \begin{subfigure}[b]{0.49\linewidth}
%     \includegraphics[width=\linewidth]{Plots/vel_51.png}
%   \end{subfigure}
%  \caption{Time evolution of the state \(\st_1(\cdot)\) and \(\st_3(\cdot)\) with \(51\) basis functions employed in \eqref{e:pcontrol}. It can be noticed that the path constraint on the states are respected. }
%  \label{fig:IP_pos_vel}
% \end{figure}
\begin{figure}[!ht]
  \centering
  \includegraphics[width=\linewidth]{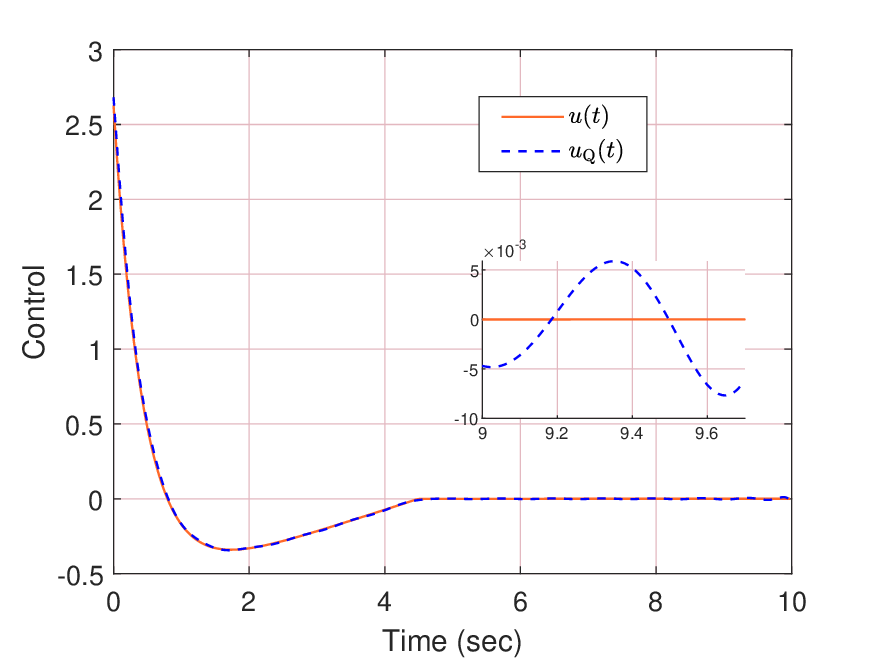}
 \caption{Control trajecories \(u(\cdot)\) and \(\cont_{\mathrm{Q}}(\cdot)\) generated by our algorithm and the direct transcription software \(\quito\), respectively. It can be seen that the control action generated by our algorithm matches closely with the one generated by \(\quito\).
 }
 \label{fig:IP_ang_angvel}
\end{figure}

\begin{figure}[ht]
\includegraphics[width=9cm,height=6cm]{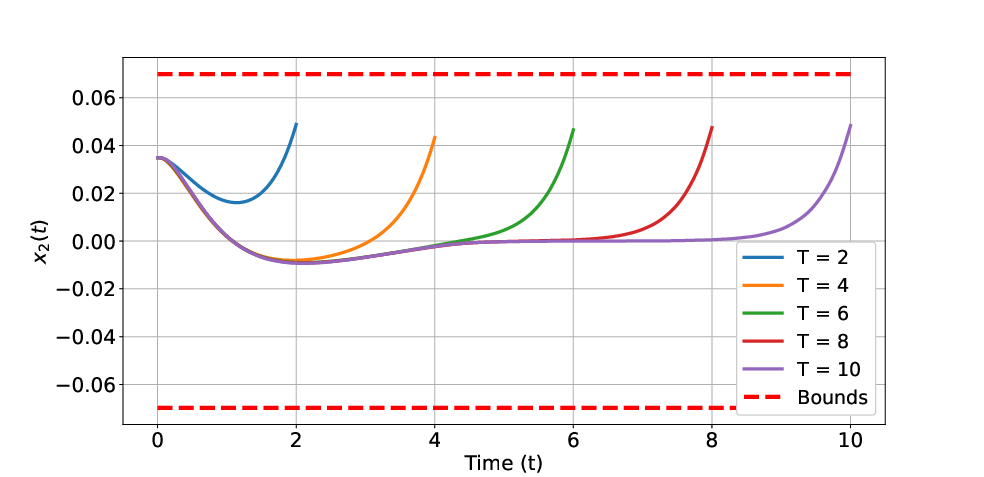}
    \caption{Time evolution of the angle made by the pendulum with the vertical for different time horizons with \(31\) basis functions in \eqref{e:pcontrol}. For each \(T\), the angle reaches equilibrium and tends towards the bound at the end of the time horizon.}
    \label{fig:IP_turnpike}
\end{figure}

\begin{table}[htbp]
\caption{A comparision of optimization statistics}
\begin{tblr}{l|c|c}
	\hline[2pt]
		\SetRow{azure9}\SetCell[c=4]{c} Solver specifications and computation time
& 3-2
& 3-3
& 3-4
\\
\hline[1pt]
Method & Solver specs & Computation time (secs) \\ 
 \hline
Ours & Dual annealing & 54.01169 \\ \hline
\(\quito\) & quadprog & 15 (with Gaussian kernel, step-size \\ & &  \(h = 0.05\), \(\mathcal{D}=2\))\\
\hline[2pt]
\end{tblr}
\centering
\label{tab:opt_data_IP}
\end{table}

\section{Conclusion}

This article introduced a numerical algorithm to solve a class of continuous-time optimal control problems in a tractable manner, with a focus on tackling model predictive control problems directly in the continuous-time regime satisfying the constraints \emph{for every time}. Two numerical examples were included which shows the numerical fidelity of the algorithm. Future work involves leveraging the tools developed herein to construct an offline explicit feedback map for a wide class of finite horizon robust optimal control problems.
\bibliographystyle{plain}
\bibliography{refs}

\end{document}